\documentclass[11pt]{article}
\usepackage[margin=1in]{geometry} 
\usepackage{amsmath,amsfonts,amsthm,alltt,amssymb}
\usepackage{xcolor}
\usepackage{graphicx}
\usepackage{float}
\usepackage{hyperref}
\usepackage{multicol}
\usepackage{listings} 
\usepackage{tikz,amstext}
\usepackage{listings} 
\definecolor{brandeisblue}{rgb}{0.0, 0.44, 1.0}

\newtheorem{theorem}{Theorem}

\newtheorem{corollary}{Corollary}

\title{ On the $p$-Schatten Energy of Bipartite Graphs}
\author{Octavio Arizmendi, José Guerrero}
\date{\today}

\begin{document}
	\maketitle
\begin{abstract}
    \noindent We give a Coulson integral formula and a Coulson-Jacobs formula for the $p$-Schatten energy. We use this formulas to compare the $p$-Schatten energy of different trees by using a quasiorder, and establish the maximality of paths among all trees.
\end{abstract}

\section{Introduction}
 For a graph $G=(V,E)$, we denote by $A_G$, the adjacency matrix of $G$, let $\lambda_1,\cdots,\lambda_n$ be its eigenvalues, and  $\phi(G,x)$ its characteristic polynomial.  The $p$-Schatten energy of $G$ is defined by the formula
$$\mathcal{E}_p(G)=\sum^n_{i=1}\lambda_i^p.$$

The purpose of this short note is to give an integral formula for the $p$-Schatten energy of a bipartite graph. In other words, we give a formula the $p$-Schatten norm of its adjacency matrix.

\begin{theorem}\label{thm:main theorem 1}
Let $G$ be a bipartite graph on $n$ vertices. For $0<p<2$, the following integral formula holds
	\begin{align}
	\label{eq:CoulsonPSchatten}
	\mathcal{E}_{p}(G) &= \frac{2\sin \left(\frac{p \pi}{2}\right)}{\pi} 
	\int_{0}^{ \infty} z^{p-1} \left(n - iz \frac{\phi'(G,iz)}{\phi(G,iz)}  \right) dz.
    \end{align}
\end{theorem}

This formula generalizes the famous Coulson integral formula \cite{Coulson40} for the energy  (corresponding to $p=1$) in the case of a bipartite graph, which has been broadly used in mathematical chemistry.  One important feature of Coulson Integral formula is that it has been proved to be very effective to compare the energies of different trees. Our formula allows to extend this feature for any $0<p<2$, in particular we use such comparison to answer the following question of Nikiforov \cite[Question 4.52.(b)]{Nik}. 
\begin{theorem}
Let $T_n$ be a tree on $n$ vertices, then For $0<p<2$, 	\begin{align}
	\label{eq:comparison1}
	\mathcal{E}_{p}(S_n) \leq \mathcal{E}_{p}(T_n)\leq \mathcal{E}_{p}(P_n).
    \end{align}
\end{theorem}
We note that this comparison theorem is in contrast with the result of Csikvari \cite{csikvari2010poset}  proving that  that for $p\in 2\mathbb{N}$,
\begin{align}
	\label{eq:comparison2}
	\mathcal{E}_{p}(S_n) \geq \mathcal{E}_{p}(T_n)\geq \mathcal{E}_{p}(P_n). 
    \end{align}
The general case $p>2$ is still open. See also Lov\'asz and Pelik\'an \cite{lovasz1973} for the case $p=\infty$.

While writing this paper we became aware of the papers \cite{Du,QuiaoLuZhang16, QiaoLuZhang17} which provide formulas which are similar to \eqref{eq:CoulsonPSchatten}. The method of proof here resembles the methods in \cite{Du}. Since the formulas in \cite{Du} have no restriction on $p$, it would be interesting to explore if their formulas allow to prove  \eqref{eq:comparison2} for all $p>2$.

\section{A Coulson integral formula}

In this section, we will give a proof of Theorem \ref{thm:main theorem 1}. That is we find a Coulson-like integral formula for the $p$-Schatten energy of a bipartite graph. We will consider the case where the number of vertices is even, the odd case follows from the even case. Indeed, on one hand, the $p$-Schatten energy does not change when adding isolated vertices, since we only add $0$'s in the spectrum. On the other hand, by noticing that, if  $G\cup\{v\}$ denotes the graph $G$ together with an isolated vertex, then
$$\frac{z(\phi(G\cup\{v\},z))'}{\phi(G\cup\{v\},z)}=\frac{z(z\phi(G,z))'}{z(\phi(G,z))}=\frac{z(z\phi(G,z)'+\phi(G,z))}{z(\phi(G,z))}=\frac{z(\phi(G,z))'}{(\phi(G,z))}+1.$$
Thus, we may deduce that the righ-hand side of \eqref{eq:CoulsonPSchatten} is also not affected by adding an isolated vertex.
 For a graph $G$, if $A_G$ is its adjacency matrix, then 
$\phi(G,z)=det(zI-A_G)$ which in terms of the eigenvalues of $A_G$ is given by $\phi(G,z)= \prod_{j=1}^n (z-\lambda_j)$. We will use the following identity which is easy to prove, $$\sum^n_{j=1} \frac{1}{z-\lambda_j}=\frac{d}{dz}[log(\phi(z))]=\frac{\phi'(G,z)}{\phi(G,z)}.$$
Now, we may proceed to prove Theorem \ref{thm:main theorem 1}. We start from the elementary identity,
	\[
	\frac{a}{it+a} - \frac{a}{it-a} = \frac{a(it-a) - a(it+a)}{(it-a)(it+a)} = \frac{-2a^2}{(it-a)(it+a)} = \frac{2a^2}{t^2+a^2}. 
	\]
	
If $G$ is a bipartite graph with $2n$ vertices,  we can write its spectrum as $\text{Spec}(G)=\{\pm \lambda_j\}_{j=1}^{n},$ for some $\lambda_i \geq 0$. From the above identity, we have,
		\begin{equation}
	\label{eq:pschatten1}
	\sum_{j=1}^{n}   \frac{2\lambda_j^2}{z^2 + \lambda_j^2} = - \sum_{j=1}^{n} \lambda_j \left(\frac{1}{iz-\lambda_j} - \frac{1}{iz+\lambda_j} \right)=  -\sum_{\lambda \in \text{Spec}(G)} \frac{\lambda}{iz-\lambda} = 2n - iz \frac{\phi'(G,iz)}{\phi(G,iz)}
	\end{equation}
where $\phi=\phi_G$.

On other hand, let us define $$f(\alpha)= \int_{0}^{\infty} \frac{t^{\alpha}}{t^2+1}dt,$$ for $\alpha \in (-1,1)$. Using a simple change of variables, we have
	\begin{equation}
		\label{eq:pschatten2}
		\int_{0}^{\infty} \frac{t^{\alpha}}{t^2+a^2}dt = \int_{0}^{\infty} \frac{a^{\alpha}(t/a)^{\alpha}}{a^2((t/a)^2+1)} dt= 
		a^{\alpha-1}\int_{0}^{\infty} \frac{s^\alpha}{s^2+1}ds = a^{\alpha-1} f(\alpha).
	\end{equation}
With the property of $f$ given in equation \eqref{eq:pschatten2},  we can integrate \eqref{eq:pschatten1} and obtain a formula involving the $(\alpha+1)$-Schatten energy

    \begin{align*}
	\int_{0}^{ \infty} z^{\alpha} \left(2n - iz \frac{\phi'(G,iz)}{\phi(G,iz)} \right) dz &= \sum_{j=1}^{n}  2\lambda_j^2 	\int_{0}^{ \infty} \frac{z^{\alpha}}{z^2 + \lambda_j^2} dz 
	= \sum_{j=1}^{n}  2\lambda_j^2  \lambda_j^{\alpha - 1}f(\alpha) \\
	&= 2 \sum_{j=1}^{n}\lambda_j^{\alpha +1} f(\alpha) 
	= f(\alpha) \mathcal{E}_{\alpha+1}(G) .
    \end{align*}

	Replacing $p = \alpha+1,$ we can write
	\begin{align}
	\label{eq:CoulsonPSchattenagain}
	\mathcal{E}_{p}(G) &= \frac{1}{f(p-1) }
	\int_{0}^{ \infty} z^{p-1} \left(2n - iz \frac{\phi'(G,iz)}{\phi(G,iz)}  \right) dz ,
    \end{align}
    and we arrive at the stated formula by realizing that  
    \[f(\alpha)=\int_{0}^{\infty} \frac{t^{\alpha}}{t^2+1}dt = \frac{\pi}{2 \cos\left(\frac{\alpha \pi}{2}\right)},\]
    which finishes the proof by the simple identity relation $\cos(t)=\sin(t+\frac{\pi}2)$.

\section{Applications}

In this section, we will extend the technique of quasi-order to compare the $p$-Schatten energy of bipartite graphs. In order to do this, recall that, if $G$ is a bipartite graph with $2n$ vertices, then  its characteristic polynomial has the following form 
\begin{equation}
\sum_{k\geq 0} (-1)^k b_{2k} x^{2n-2k},
\end{equation}
where $b_{2k}\geq 0$ for all $k$. The quasi-order $\preceq$ is defined for bipartite graphs as follows: $G_1 \preceq G_2$ if $b_{2k}(G_1)\leq b_{2k} (G_2)$ for all $k$, see \cite{GZ,Zh} or Section 4.3 of \cite{Glibro}. 

We note that, for $z\in \mathbb{R}$ we have
\begin{align*}
\phi(G, iz) &= \sum_{k\geq 0} (-1)^k b_{2k} (iz)^{2n-2k}=\sum_{k\geq 0} b_{2k} z^{2n-2k}
\end{align*}
and thus $\phi(G, iz)>0$ and $\phi(G_2, iz)\leq\phi(G_1, iz)$ if $G_2\preceq G_1$, for $z\in [0, \infty)$.

Our aim is to use the integral formula of Theorem \ref{thm:main theorem 1}, to compare $p$-Schatten energy of two graphs, $G_1$ and $G_2$. For this, we will need to modify \eqref{eq:CoulsonPSchatten} as done in classical paper of Coulson and Jacobs \cite{coulsonjacobs}

\begin{theorem}
[Coulson-Jacobs formula for $p$-Energy]

Let $G_1$ and $G_2$ be bipartite graphs on $2n$ vertices, and $0<p<2$, then
	\begin{align}
	\label{eq:CoulsonJacobs}
	\mathcal{E}_{p}(G_1) -\mathcal{E}_{p}(G_2) &= \frac{2p \sin \left(\frac{p \pi}{2}\right)}{\pi}
	\int_{0}^{ \infty} z^{p-1} \log\left(\frac{\phi(G_1,iz)}{\phi(G_2,iz)} \right) dz.
    \end{align}
\end{theorem}

\begin{proof}
From Theorem \ref{thm:main theorem 1} we see that 
	\begin{align*}
	\mathcal{E}_{p}(G_1)-	\mathcal{E}_{p}(G_2) 
&= \frac{2\sin \left(\frac{p \pi}{2}\right)}{\pi} 
\int_{0}^{ \infty} iz^{p}\left( \frac{\phi'(G_2,iz)}{\phi(G_2,iz)} - \frac{\phi'(G_1,iz)}{\phi(G_1,iz)}  \right) dz  \\
& = \frac{2\sin \left(\frac{p \pi}{2}\right)}{\pi} 
\int_{0}^{ \infty} z^{p} \frac{d}{dz}\log\left(\frac{\phi(G_2,iz)}{\phi(G_1,iz)}  \right) dz   \\
&= -\frac{2\sin \left(\frac{p \pi}{2}\right)}{\pi} 
\int_{0}^{ \infty} p z^{p-1} \log\left(\frac{\phi(G_2,iz)}{\phi(G_1,iz)}  \right) dz
\\&= \frac{2p\sin \left(\frac{p \pi}{2}\right)}{\pi} 
\int_{0}^{ \infty}  z^{p-1} \log\left(\frac{\phi(G_1,iz)}{\phi(G_2,iz)}  \right) dz
\end{align*}
where we used integration by parts in the second to last equality.   \end{proof}

\begin{corollary}
Let $G_1$ and $G_2$ be two bipartite graphs. If $G_1\preceq G_2$ then $\mathcal{E}(G_1)\leq\mathcal{E}(G_2)$. 
\end{corollary}
\begin{proof}
Notice that if $G_1\preceq G_2$ then $\phi(G_2, iz)\geq \phi(G_1, iz)$ for all $z\in \mathbb{R}$ and then $log\left(\frac{\phi(G_2, iz)}{\phi(G_1, iz)}\right)\geq0,$ from where we see that
\begin{align}
	\mathcal{E}_{p}(G_2) -\mathcal{E}_{p}(G_1) &= \frac{2p \sin \left(\frac{p \pi}{2}\right)}{\pi}
	\int_{0}^{ \infty} z^{p-1} log\left(\frac{\phi(G_2,iz)}{\phi(G_1,iz)} \right) dz\geq0,
    \end{align}
since the above integral must be positive.
\end{proof}

Now, Theorem 2 is a direct consequence of the following known relation in the quasi-order for trees (see for example \cite{GZ} or Theorem 4.6 in \cite{Glibro}).

\begin{theorem}
Let $T_n$ be a tree on $n$ vertices, then 	\begin{align}
		S_n \preceq T_n\preceq P_n.
    \end{align}
\end{theorem}

\bibliographystyle{plain}
\bibliography{bibCoulsonPSchatten} 




\end{document}